\documentclass{amsart}
\usepackage[active]{srcltx}
\usepackage{calc,amssymb,amsthm,amsmath,amscd, eucal,ulem}
\usepackage{alltt}
\usepackage[left=1.35in,top=1.25in,right=1.35in,bottom=1.25in]{geometry}
\RequirePackage[dvipsnames,usenames]{color}

\input{kmacros3.sty}
\input{xy}
\xyoption{all}
\usepackage{tikz}

\numberwithin{equation}{theorem}

\renewcommand{\m}{\mathfrak{m}}
\renewcommand{\n}{\mathfrak{n}}
\DeclareMathOperator{\Tr}{Tr}

\DeclareMathOperator{\fp}{fpt}
\DeclareMathOperator{\lct}{lct}
\usepackage{fullpage}

\usepackage{setspace}
\usepackage{hyperref}
\usepackage{enumerate}
\usepackage{graphicx}

\usepackage[all,cmtip]{xy}
\usepackage{verbatim}

\theoremstyle{theorem}

\begin{document}

\title{$F$-singularities under generic linkage}
\author{Linquan Ma, Janet Page, Rebecca R.G., William Taylor, and Wenliang Zhang}
\address{Department of Mathematics\\ University of Utah\\ Salt Lake City\\ UT 84112}
\email{lquanma@math.utah.edu}
\address{Department of Mathematics\\ University of Illinois at Chicago\\ Chicago\\ Illinois 60607}
\email{jpage8@uic.edu}
\address{Department of Mathematics\\ Syracuse University \\ Syracuse\\ NY 13244}
\email{rirebhuh@syr.edu}
\address{Department of Mathematical Sciences \\ University of Arkansas \\ Fayetteville \\ AR 72701}
\email{wdtaylor@uark.edu}
\address{Department of Mathematics\\ University of Illinois at Chicago\\ Chicago\\ Illinois 60607}
\email{wlzhang@uic.edu}

%\thanks{L. Ma was partially supported by NSF CAREER Grant DMS \#1252860/1501102 and a Simons Travel Grant. }
%\thanks{W. Zhang was partially supported by the NSF grant DMS \#1606414.}
\maketitle

\begin{center}
{\textit{Dedicated to Prof. Craig Huneke on the occasion of his 65th birthday}}
\end{center}

\begin{abstract}
Let $R=k[x_1,\dots,x_n]$ be a polynomial ring over a prefect field of positive characteristic. Let $I$ be an equi-dimensional ideal in $R$ and let $J$ be a generic link of $I$ in $S=R[u_{ij}]_{c \times r}$. We describe the parameter test submodule of $S/J$ in terms of the test ideal of the pair $(R, I)$ when a reduction of $I$ is a complete intersection or almost complete intersection. As an application, we deduce a criterion for when $S/J$ has $F$-rational singularities in these cases. We also compare the $F$-pure threshold of $(R, I)$ and $(S, J)$.
\end{abstract}

%%%%%%%%%%%%%%%%%%%%%%%%%%%%%%%%%%%%%%%%%%%%%%%%%%%%%%%%%%%%%%%%%%%%%
\section{Introduction}
%%%%%%%%%%%%%%%%%%%%%%%%%%%%%%%%%%%%%%%%%%%%%%%%%%%%%%%%%%%%%%%%%%%%%

Let $R=k[x_1,\dots,x_n]$ be a polynomial ring over a field of positive characteristic. Let $I=(f_1,\dots, f_r)$ be an equi-dimensional ideal in $R$ of height $c$, where equi-dimensional means that all associated primes of $I$ have the same height \cite{Matsumura86}. We can define a regular sequence $g_1,\dots, g_c$ in $S=R[u_{ij}]_{c \times r}$ via $g_i:=u_{i1}f_1+\cdots +u_{ir}f_r$, where the $u_{ij}$  are variables over $S$. Then $J=(g_1,\dots,g_c):I$ is called a generic link of $I$ in $S=R[u_{ij}]$. The study of generic linkage has attracted considerable attention and has been developed widely from both algebraic and geometric points of view \cite{HunekeUlrichThestructureoflinkage}, \cite{HunekeUlrichAlgebraiclinkage}, \cite{ChardinUlrichLiaisonandCMregularity}, \cite{EisenbudHunekeUlrichHeightsofidealsofminors}, \cite{NiuSingularitiesofgenericlinkage}.

In contrast to the quick and deep development of singularity theories in the past decades, much less has been known about the behaviors of singularities under generic linkage. A special case is a result of Chardin and Ulrich \cite{ChardinUlrichLiaisonandCMregularity} which says that if $R/I$ is a complete intersection and has rational (resp. $F$-rational) singularities, then a generic link $S/J$ also has rational (resp. $F$-rational singularities). This result in characteristic zero has been vastly extended in recent work of Niu \cite{NiuSingularitiesofgenericlinkage}, which is our main motivation for this research.

\begin{theorem}[Theorem 1.1 in \cite{NiuSingularitiesofgenericlinkage}]
\label{theorem: Niu}
Let $J$ be a generic link of a reduced and equidimensional ideal $I$ in $S=R[u_{ij}]$ and assume that the characteristic of $k$ is $0$. We have
\begin{enumerate}
\item $\omega_{S/J}^{GR}\cong \scr{J}(R, I^c)\cdot (S/J)$, where $\omega_{S/J}^{GR}$ denotes the Grauert-Riemenschneider canonical sheaf of $S/J$ and $\scr{J}(R, I^c)$ denotes the multiplier ideal of the pair $(R, I^c)$,
\item $\lct(S, J)\geq \lct(R, I)$. In particular, if the pair $(R, I^c)$ is log canonical, then the pair $(S, J^c)$ is also log canonical.
\end{enumerate}
\end{theorem}

This result gives a nice criterion for a generic link to have rational singularities in characteristic $0$. It also has applications to bounding the Castelnuovo-Mumford regularity of projective varieties \cite[Corollary 1.2]{NiuSingularitiesofgenericlinkage}. Since test ideals and $F$-pure thresholds are characteristic $p$ analogues of multiplier ideals and log canonical thresholds ({\it c.f.} \cite{BlickleSchwedeTuckerF-singularitiesvialterations} and \cite{HaraYoshidaGeneralizedTestIdeals}), it is natural to ask whether analogues of Theorem \ref{theorem: Niu} hold for test ideals and $F$-pure thresholds. Our main result is the following, which partially extends Theorem \ref{theorem: Niu} to characteristic $p$ and generalizes \cite[Theorem 3.13]{ChardinUlrichLiaisonandCMregularity} in characteristic $p$.

\begin{theorem}[Theorem \ref{theorem--parameter test submodules under generic linkage}, Corollary \ref{corollary -- generic link lower bound}]
Let $J$ be a generic link of an equi-dimensional ideal $I$ in $S=R[u_{ij}]$ and assume the characteristic of $k$ is $p>0$.
\begin{enumerate}
\item Suppose $I$ is reduced and that a reduction of $I$ is a complete intersection or an almost complete intersection. Then $\tau(\omega_{S/J}) \cong\tau(R, I^c)\cdot(S/J)$, where $\tau(\omega_{S/J})$ denotes the parameter test submodule and $\tau(R, I^c)$ denotes the test ideal of the pair $(R, I^c)$.
\item Suppose that a reduction of $I$ is a complete intersection. Then $\fp_S(J)\geq \fp_R(I)$. In particular, if the pair $(R, I^c)$ is $F$-pure, then the pair $(S, J^c)$ is also $F$-pure.
\end{enumerate}
\end{theorem}

%To the best of our knowledge, this is the first result that explores the connections between linkage theory and singularities in positive characteristic.

This paper is organized as follows: in Section 2 we recall and prove some basic result for $F$-singularities and test ideals; in Section 3 we give a description of the parameter test submodule of $S/J$ in terms of the test ideal of the pair $(R, I)$, when a reduction of $I$ is a complete intersection or an almost complete intersection. This generalizes earlier results in \cite{ChardinUlrichLiaisonandCMregularity}. In Section 4 we compare the $F$-pure threshold of the pairs $(S, J)$ and $(R, I)$ when a reduction of $I$ is a complete intersection.

%\vskip 9pt
%\noindent
\subsection*{Acknowledgements}  Part of this work was done at Mathematics Research Community (MRC) in Commutative Algebra in June 2015. The authors would like to thank the staff and organizers of the MRC and the American Mathematical Society for their support. The first author would like to thank Karl Schwede, Shunsuke Takagi and Bernd Ulrich for fruitful discussions. The first author was partially supported by NSF Grant DMS \#1600198, NSF CAREER Grant DMS \#1252860/1501102 and a Simons Travel Grant when preparing this article. The second author was partially supported by NSF RTG grant DMS
\#1246844. The third author was partially supported by the NSF grant DGE \#1256260. The fifth author was partially supported by the NSF grant DMS \#1606414. The authors thank the referee for some comments which lead to improvement of the presentation of the paper.

\section{$F$-singularities and test ideals}
\label{section: basics on test ideals}

In this section we collect some basic definitions of $F$-singularities and test ideals and prove a characteristic $p>0$ analogue of Ein's Lemma in \cite{NiuSingularitiesofgenericlinkage}, which will be used in later sections.

Let $R$ be a Noetherian commutative ring of characteristic $p>0$. We will use $F^e_*R$ to denote the target of the $e$-th Frobenius endomorphism $F^e:R\xrightarrow{r\mapsto r^{p^e}}R$, {\it i.e.} $F^e_*R$ is an $R$-bimodule, which is the same as $R$ as an abelian group and as a right $R$-module, that acquires its left $R$-module structure via the $e$-th Frobenius endomorphism $F^e:R\xrightarrow{r\mapsto r^{p^e}}R$. When $R$ is reduced, we will use $R^{1/p^e}$ to denote the ring whose elements are $p^e$-th roots of elements of $R$. Note that these notations (when $R$ is reduced) $F^e_*R$ and $R^{1/p^e}$ are used interchangeably in the literature; we will do so in this paper as well assuming no confusion will arise.

\begin{remark}
If $R$ is a commutative ring essentially of finite type over a perfect field of characteristic $p>0$, then $R$ admits a canonical module denoted by $\omega_R$. Applying $\Hom_R(-,\omega_R)$ to the $e$-th Frobenius $R \to F^e_*R$ produces an $R$-linear map \[\Hom_R(F^e_*R,\omega_R)\to\Hom_R(R,\omega_R)=\omega_R.\] Moreover, we have $F_*^e\omega_R\cong \Hom_R(F^e_*R,\omega_R)$ (see \cite[Example 2.4]{BlickleSchwedeTuckerF-singularitiesvialterations} for more details). Hence we have a natural $R$-linear map:
\[\Phi_R^e: F_*^e\omega_R\cong \Hom_R(F^e_*R,\omega_R)\to\Hom_R(R,\omega_R)=\omega_R\]
called the trace map of the $e$-th Frobenius.
\end{remark}

\begin{example}
When $R=k[x_1,\dots,x_n]$ is a polynomial ring over a perfect field $k$ of characteristic $p>0$, we can identify $\omega_R$ with $R$, and $\Phi_R^e$ can be identified with the usual trace $\Tr_R^e$, that is:
$$\Tr_R^e(F_*^e(x_1^{i_1}x_2^{i_2}\cdots x_n^{i_n}))=\begin{cases}x_1^{\frac{i_1-(p^e-1)}{p^e}}x_2^{\frac{i_2-(p^e-1)}{p^e}}\cdots x_n^{\frac{i_n-(p^e-1)}{p^e}}, & {\rm if\ }\frac{i_t-(p^e-1)}{p^e}\in \mathbb{Z}\ {\rm for\ each\ }t\\0, & {\rm otherwise} \end{cases}$$
In this case $\Hom_R(F_*^eR, R)$ is a cyclic $F_*^eR$-module generated by $\Tr_R^e$. Furthermore, if $f_1,\dots, f_c$ is a regular sequence in $R$ and $T=R/(f_1,\dots,f_c)$, then we have (\cite[Corollary on page 465]{FedderFPureRat}\footnote{Fedder's result \cite[Corollary on page 465]{FedderFPureRat} assumes that the ring $R$ is a Gorenstein local ring only to ensure that $\Hom_R(F_*R,R)\cong F_*R$. In our case, $R=k[x_1,\dots,x_n]$ is a polynomial ring, so $\Hom_R(F_*R,R)$ is clearly isomorphic to $F_*R$. Hence Fedder's result applies in our case.})
\[\Phi_T^e(F_*^e(-))= \text{image of } \Tr_R^e(F_*^e(f_1^{p^e-1}\cdots f_c^{p^e-1}\cdot -)) \text{ in $T$}.\]
\end{example}

\begin{remark}
\label{rmk: trace does not change under polynomial ext}
Let $R=k[x_1,\dots,x_n]$ be a polynomial ring over a field $k$ of characteristic $p>0$ and $A=R[y_1,\dots,y_m]$ be a polynomial ring over $R$. For each ideal $I$ in $R$, it is well known and straightforward to check that
\[\Tr_R^e(F_*^e(I))A=\Tr_A^e(F_*^e(IA)).\]
\end{remark}

\begin{lemma}
\label{lemma--restriction of trace}
Let $S\to R$ be a surjection of Noetherian commutative rings of characteristic $p$. Assume that both $S$ and $R$ admit canonical module $\omega_S$  and $\omega_R$ respectively and $\dim S=\dim R$. Then
\[\Phi_R^e=\Phi_S^e|_{\omega_R}.\]
%where we identify $\omega_R$ as a submodule of $\omega_S$ under $\omega_R\cong\Hom_S(R, \omega_S)\hookrightarrow \omega_S$.
\end{lemma}
\begin{proof}
Under our assumptions, we have $\omega_R=\Hom_S(R,\omega_S)$ and the surjection $S\to R$ induces an inclusion $\omega_R=\Hom_S(R,\omega_S)\hookrightarrow \omega_S$. Consider the following diagram
\[\xymatrix{
\Hom_R(F^e_*R,\Hom_S(R,\omega_S)) \ar[d] \ar[r] & \Hom_R(R,\Hom_S(R,\omega_S)) \ar[r]^-{\sim} \ar[d]& \Hom_S(R,\omega_S)\ar[d] \\
\Hom_S(F^e_*S,\omega_S) \ar[r] &\Hom_S(S,\omega_S) \ar[r]^-{\sim} &\omega_S
}\]
Note that the top row (resp. the bottom row) induces the trace map $\Phi_R^e$ (resp. $\Phi_S^e$). To prove our lemma, it suffices to prove
\begin{enumerate}
\item[(a)] the vertical map on the left is an inclusion, and
\item[(b)] the diagram commutes
\end{enumerate}

To prove (a), note that the vertical map on the left can be refined further as
\begin{align}
\Hom_R(F^e_*R,\Hom_S(R,\omega_S)) &= \Hom_S(F^e_*R,\Hom_S(R,\omega_S))\notag\\
&\hookrightarrow \Hom_S(F^e_*S,\Hom_S(R,\omega_S))\ {\rm since\ }F^e_*S\twoheadrightarrow F^e_*R\notag\\
&\hookrightarrow \Hom_S(F^e_*S,\omega_S)\ {\rm since\ }\Hom_S(R,\omega_S)\hookrightarrow \omega_S\notag
\end{align}

To prove (b), note that the commutativity follows directly from the commutativity of
\[\xymatrix{
S\ar[r] \ar[d] & F^e_*S \ar[d]\\
R\ar[r] &F^e_*R
}\]
\end{proof}

\begin{definition}[{\it cf.} Definition 3.1 in \cite{HaraGeometricinterpretation} and Definition 2.33 in \cite{BlickleSchwedeTuckerF-singularitiesvialterations}]
\label{definition--parameter test submodule}
Let $R$ be an $F$-finite Noetherian integral domain of characteristic $p$. The \textit{parameter test submodule} $\tau(\omega_R)$ is the unique smallest nonzero submodule $M$ of $\omega_R$ such that $\Phi_R(F_*M)\subseteq M$. $R$ is called \textit{$F$-rational} if $R$ is Cohen-Macaulay and $\tau(\omega_R)=\omega_R$. Note that this is not the original definition of $F$-rationality, but is known to be equivalent \cite{SmithFRatImpliesRat}.
\end{definition}

\begin{definition}[{\it cf.} Definition 3.16 and Theorem 3.18 in \cite{SchwedeTestidealsinnonQ-Gorensteinrings}]
\label{definition--test ideals}
Let $R$ be an $F$-finite Noetherian integral domain of characteristic $p$. Let $I\subseteq R$ be a nonzero ideal and $t\in \Q_{\geq 0}$. We define the \textit{test ideal} $\tau(R, I^t)$, abbreviated $\tau(I^t)$, to be the unique smallest nonzero ideal $J\subseteq R$ such that $\phi(F^e_*(I^{\lceil t(p^e-1)\rceil}J))\subseteq J$ for all $e>0$ and all $\phi\in\Hom_R(F^e_*R, R)$.
\end{definition}

%\todo{definition of $F$-pure pairs, $F$-regular pairs and $F$-pure thresholds}
%\Taylor{There are some refinements of these concepts that Karl suggests using in his paper "A Refinement of Sharply $F$-pure and Strongly $F$-regular pairs".  Do we want to consider altering our definitions to match his?  I don't know whether it will affect much in what we're doing.}
\begin{definition}[{\it cf.} Definitions 1.3 and 2.1 and Proposition 2.2 in \cite{TakagiWatanabeFpurethresholds}]
\label{definition--F-pure pairs,strongly F-regular pairs, F-pure threshholds}
Let $R$ be an $F$-finite, local, Noetherian, integral domain of characteristic $p$. Let $I \subset R$ be an ideal and $t \geq 0$ be a real number.
\begin{enumerate}
\item The pair $(R,I^t)$ is \textit{$F$-pure} if for all large $e \gg 0$, there exists an element $d \in I^{\lfloor t(p^e-1) \rfloor}$ such that $(F^e_*d) R \hookrightarrow F^e_*R$ splits as an $R$-module homomorphism.
%\item $(R,I^t)$ is strongly $F$-pure if there exist $e\geq 0$ and $d\in I^{\lceil tp^e\rceil}$ such that $F^e_*d R \hookrightarrow F^e_* R$ splits as an $R$-module homomorphism.
\item The pair $(R, I^t)$ is \textit{strongly $F$-regular} if for every $c\neq 0$ there exists $e \geq 0$ and $d \in I^{\lceil tp^e \rceil}$ such that $F^e_*(cd)R \hookrightarrow F^e_*R$ splits as an $R$-module homomorphism.
\item The \textit{$F$-pure threshold} $\fp_R(I)$ of $(R,I)$ is $\sup \{ s \in \mathbb{R}_{\geq 0} | \text{ the pair } (R,I^{s}) \text{ is $F$-pure}\}$, and when $R$ is strongly $F$-regular, we also have $\fp_R (I)=\sup \{ s \in \mathbb{R}_{\geq 0} | \text{ the pair } (R,I^{s}) \text{ is strongly $F$-regular}\}$.
\end{enumerate}
\end{definition}

\begin{remark}
\label{remark -- strong F-regularity test ideal}
Note that when $R$ is local, $(R,I^t)$ is strongly $F$-regular if and only if  $\tau(I^t)=R$.  Indeed, suppose $(R,I^t)$ is strongly $F$-regular.  Pick a nonzero element $c \in J$ and take $e\gg 0$ and $d\in I^{\lceil tp^e\rceil}$ satisfying the conditions of strong $F$-regularity for $c$, and let $\phi:F^e_*R\to R$ be a map such that $\phi(F^e_*(cd))=1$.  Then \[\phi(F^e_*(I^{\lceil t(p^e-1)\rceil}J))\supseteq \phi(F^e_*(I^{\lceil tp^e\rceil}J))=R,\] and so $\tau(I^t)=R$.

On the other hand, if $\tau(I^t)=R$, $0\neq c\in R$, and $a\in I^{\lceil t\rceil}$, then there exists $e\geq 0$ and $\phi:F^e_*R\to R$ such that $\phi(F^e_*(I^{\lceil t(p^e-1)\rceil}acR))=R$.  Let $b\in   I^{\lceil t(p^e-1)\rceil}$ and $f\in R$ such that $\phi(F^e_*(c(abf)))=1$.  Then we are done once we note that $abf\in I^{\lceil t\rceil}I^{\lceil t(p^e-1)\rceil}\subseteq I^{\lceil tp^e\rceil}$.
\end{remark}

We will need the following important description of test ideals:
\begin{theorem}[{\it cf.} Proof of Theorem 3.18 in \cite{SchwedeTestidealsinnonQ-Gorensteinrings}]
\label{theorem--description of test ideals}
With the notations as in Definition \ref{definition--test ideals}, for any nonzero $a\in\tau(I^t)$, we have:
$$\tau(I^t)=\sum_{e\geq 0}\sum_{\phi}\phi(F_*^e(aI^{\lceil t(p^e-1)\rceil}))$$
where the inner sum runs over all $\phi\in\Hom_R(F_*^eR, R)$.
\end{theorem}

\begin{remark}
\label{remark--choose test element}
With the notations as in Definition \ref{definition--test ideals}, the following holds (\cite[3.3]{BlickleSchwedeTakagiZhangDiscretenessRationality})
\begin{equation}
\label{test ideal using big test element}
\tau(I^t)=\sum_{e\geq 0}\sum_{\phi\in \Hom_R(F^e_*R,R)}\phi(F^e_*(dI^{\lceil tp^e \rceil }))
\end{equation}
where $d$ is a big test element (which is just a nonzero element in $\tau(R)=\tau(R, I^0)$).

If $R=k[x_1,\dots,x_n]$ is a polynomial ring over a perfect field $k$ of characteristic $p>0$, then one can set $d=1$ in (\ref{test ideal using big test element}) and $\Hom_R(F_*^eR, R)$ is a cyclic $F^e_*R$-module generated by $\Tr_R^e$ as discussed earlier. Hence by (\ref{test ideal using big test element}),
\begin{align}
\tau(I^t)&=\sum_{e\geq 0}\sum_{\phi\in \Hom_R(F^e_*R,R)}\phi(F_*^e(aI^{\lceil t(p^e-1)\rceil}))=\sum_{e\geq 0}\Tr_R^e(F_*^e(aI^{\lceil t(p^e-1)\rceil})),\ {\rm for\ any\ }a\in\tau(I^t)\notag\\
&=\sum_{e\geq 0}\sum_{\phi\in \Hom_R(F^e_*R,R)}\phi(F^e_*(I^{\lceil tp^e \rceil }))=\sum_{e\geq 0}\Tr_R^e(F^e_*(I^{\lceil tp^e \rceil }))\notag
\end{align}
\end{remark}

\begin{remark}
\label{remark--description of parameter test submodule}
With the notations as in Definition \ref{definition--parameter test submodule}, one can show that if $R_{a'}$ is regular, then for every sufficiently large power $a$ of $a'$, $\tau(\omega_R)=\sum_e\Phi_R^e(F^e_*(a\cdot \omega_R))$. This can be proved by a similar argument as \cite[Lemma 3.6, Lemma 3.8]{SchwedeTuckerAsurveyoftestideals} so we omit the details.
\end{remark}

The following result from \cite{SchwedeTuckerAsurveyoftestideals} will also be used. These results were originally proved in \cite{HaraYoshidaGeneralizedTestIdeals} and \cite{HaraTakagiOnGeneralizationTestideals}, and they hold as long as $R$ is $F$-finite. We will only state the version of these results that we need.

\begin{lemma}[{\it cf.} Theorem 6.9 in \cite{SchwedeTuckerAsurveyoftestideals}]
Let $R$ be an integral domain essentially of finite type over a perfect field of characteristic $p>0$ and let $I,J\subseteq R$ be nonzero ideals and $t\in \R_{\geq 0}$.
\begin{enumerate}
\item If $J$ is a reduction of $I$, then $\tau(I^t)=\tau(J^t)$.
\item If $J$ is generated by $r$ elements, then $\tau(J^r)=J\tau(J^{r-1})$.
\end{enumerate}
\end{lemma}

We are ready to prove the characteristic $p>0$ analogue of Ein's Lemma in \cite{NiuSingularitiesofgenericlinkage}:

\begin{lemma}[Ein's Lemma in characteristic $p>0$]
\label{lemma--Ein's lemma in char p}
Let $R$ be an integral domain essentially of finite type over a perfect field of characteristic $p>0$ and let $I\subseteq R$ be an equi-dimensional and unmixed ideal of codimension $c$.  If $\tau(I^{c-1})=R$, then $\tau(I^c)=I$.  In particular, if $R$ is strongly $F$-regular and $(R,I^c)$ is $F$-pure, then $\tau(I^c)=I$.
\end{lemma}
\begin{proof}
The lemma will follow from the following two inclusions:
\begin{equation}
\label{equation--containment of test ideal 1}
\tau(I^c)\subseteq I.
\end{equation}
\begin{equation}
\label{equation--containment of test ideal 2}
I\tau(I^{t-1})\subseteq \tau(I^t) \mbox{ for all } t\geq 1.
\end{equation}
Indeed,  if $\tau(I^{c-1})=R$, then $I=I\tau(I^{c-1})\subseteq \tau(I^c)\subseteq I$, and so we have equality throughout.

%\vspace{1em}

\begin{proof}[Proof of (\ref{equation--containment of test ideal 1})]
Since inclusion is a local condition, we may assume that $R$ is local with maximal ideal $\mathfrak{m}$. By replacing $R$ by $R[x]_{\mathfrak{m}R[x]}$, we may assume that $R$ has infinite residue field: it is straightforward to check that $\tau(I^c)R[x]_{\mathfrak{m}R[x]}=\tau((IR[x]_{\mathfrak{m}R[x]})^c)$. Now let $\mathfrak{p}$ be a minimal prime of $I$.  Since $I$ is equi-dimensional, $\dim R_\mathfrak{p}=c$.  Hence $IR_\mathfrak{p}$ has a reduction $J\subseteq IR_\mathfrak{p}$ generated by $c$ elements.  Therefore, since test ideals localize,
\[\tau(I^c)R_\mathfrak{p}=\tau((IR_\mathfrak{p})^c)=\tau(J^c)=J\tau(J^{c-1})\subseteq J\subseteq IR_\mathfrak{p}.\]
Since every associated prime of $I$ is minimal, this inclusion holds for all associated primes of $I$, hence it holds globally, i.e.\ $\tau(I^c)\subseteq I$.
\end{proof}

\begin{proof}[Proof of (\ref{equation--containment of test ideal 2})] This should be well known to experts in the field; we opt to provide a proof here since we could not locate a proper reference.
Let $t\in \R_{\geq 1}$, and pick $0\neq a\in \tau(I^t)$.
Then
\begin{align*}
I\tau(I^{t-1}) &= I\sum_{e\geq 0}\sum_{\phi}\phi\left(F_*^e\left(aI^{\lceil (t-1)(p^e-1)\rceil}\right)\right)\\
&=\sum_{e\geq 0}\sum_{\phi}\phi\left(F_*^e\left(aI^{[p^e]}I^{\lceil (t-1)(p^e-1)\rceil}\right)\right)\\
&\subseteq \sum_{e\geq 0}\sum_{\phi}\phi\left(F_*^e\left(aI^{p^e}I^{\lceil (t-1)(p^e-1)\rceil}\right)\right)\\
&=\sum_{e\geq 0}\sum_{\phi}\phi\left(F_*^e\left(aI^{p^e+\lceil (t-1)(p^e-1)\rceil}\right)\right)\\
&\subseteq \sum_{e\geq 0}\sum_{\phi}\phi\left(F_*^e\left(aI^{\lceil t(p^e-1)\rceil}\right)\right)\\
&=\tau(I^t),
\end{align*}
where the inner sum runs over all $\phi\in\Hom_R(F_*^eR, R)$ and the last inclusion following from the fact that
\[p^e+\lceil(t-1)(p^e-1)\rceil=\lceil p^e+(t-1)(p^e-1)\rceil=\lceil t(p^e-1)+1\rceil>\lceil t(p^e-1)\rceil.\qedhere\]
\end{proof}

For the last statement, if $(R,I^c)$ is $F$-pure, then the $F$-pure threshold of $I$ is at least $c$.  Since the $F$-pure threshold is the supremum of those values $t$ for which $(R,I^t)$ is strongly $F$-regular when $R$ is strongly $F$-regular \cite[Proposition 2.2]{TakagiWatanabeFpurethresholds}, we have that $(R,I^{c-1})$ is strongly $F$-regular. This means that $\tau(I^{c-1})=R$ by Remark \ref{remark -- strong F-regularity test ideal}, and hence the first statement of the lemma tells us $\tau(I^c)=I$.
\end{proof}

%We prove a partial converse to Ein's lemma; it is unknown whether the converse holds in greater generality.

%\begin{lemma}
%Let $R$ be an local domain essentially of finite type over a perfect field of characteristic $p>0$ and let $I\subseteq R$ be an equi-dimensional nonzero ideal of codimension $c$ such that either $I$ is a complete intersection or $c=\dim R$.  Suppose that $\tau(I^c)=I$. Then $\tau(I^{c-1})=R$. %I added the word nonzero here--we do need this for the argument used in the proof. If the ideal is zero, this is boring anyway.
%\end{lemma}

%\begin{proof}
%In either case of our hypothesis, $I$ has a reduction $J$  generated by $c$ elements, so
%\[I=\tau(I^c)=\tau(J^c)=J\tau(J^{c-1})\subseteq I\tau(I^{c-1})\subseteq I,\]
%so $I=I\tau(I^{c-1})$. This forces $\tau(I^{c-1})=R$ by Nakayama's lemma.
%\end{proof}

\section{$F$-rationality under generic linkage}
In this section, we investigate how $F$-singularities ({\it e.g.} $F$-purity, $F$-rationality, etc) behave under a generic linkage. To this end, we will also consider the behaviors of test ideals under a generic linkage. We begin with recalling the definition of a generic link.
\begin{definition}
\label{definition--generic linkage}
Let $R=k[x_1,\dots, x_n]$ be a polynomial ring over a perfect field of positive characteristic. Let $I$ be an equi-dimensional and unmixed ideal of $R$ of height $c$. Fix a generating set $\{f_1,\dots, f_r\}$ of $I$. Let $(u_{ij})$, $1\leq i\leq c$, $1\leq j\leq r$, be a $c\times r$ matrix of variables.  Consider $c$ elements $g_1,\dots,g_c$ in $S=R[u_{ij}]$ defined by \[g_i:=u_{i1}f_1+u_{i2}f_2+\cdots +u_{ir}f_r\] for $1\leq i\leq c$. Then $J=(g_1,\dots,g_c): (IS)$ is called the {\it first generic link} of $I$ with respect to $\{f_1,\dots, f_r\}$ (we also call $S/J$ the generic link of $R/I$ with respect to $\{f_1,\dots, f_r\}$).
\end{definition}
%\Zhang{I removed `reduced' from this definition to accommodate Proposition \ref{proposition -- m-primary monomial ideal}. One consequence is that the proof of Corollary \ref{corollary--criterion for F-rationality} is now different; another consequence is that we need to remove the next remark.}
\begin{remark}
\label{remark--geometrically linked}
It is well known that under the above assumptions, if $I$ is reduced, then $IS$ and $J$ are {\it geometrically linked}, i.e., $IS=(g_1,\dots,g_c):J$ and $IS\cap J=(g_1,\dots,g_c)$. Moreover, $J$ is actually a prime ideal of height $c$ \cite[2.6]{HunekeUlrichDivisorclassgroupsdeformations}.
\end{remark}

The following theorem is our main technical result in this section.
\begin{theorem}
\label{theorem--parameter test submodules under generic linkage}
With the notation as in Definition \ref{definition--generic linkage}, assuming $I$ is reduced, we have
\begin{enumerate}
\item $\tau(\omega_{S/J})\subseteq \tau(I^c)\cdot(S/J)$;
\item If $I$ has a minimal reduction generated by at most $c+1$ elements, then $\tau(\omega_{S/J})\supseteq \tau(I^c)\cdot(S/J)$; hence $\tau(\omega_{S/J})=\tau(I^c)\cdot(S/J)$ in this case.
\end{enumerate}
\end{theorem}

%\Ma{I add ``$I$ reduced" in Remark 3.2. I strongly suggest we add ``$I$ reduced" in this theorem, although I believe the result should hold as long as $I$ is equi-dimensional. This is because if $I$ is not reduced (or in general not a generic complete intersection), then $S/J$ is not necessarily a domain. Hence the way we define test ideal and test submodule and the references we cited should all be modified to fit in this more general context. Our current set-up in section 2 is ``domain essentially of finite type over $k$".}

Our proof of Theorem \ref{theorem--parameter test submodules under generic linkage}(2) requires considering different sets of generators of $I$. A priori, a generic link $(S,J)$ depends on the choice of generators. The following lemma guarantees that the statement in Theorem \ref{theorem--parameter test submodules under generic linkage}(2) is independent of the choice of generators of $I$. Its proof follows the same line as the proof of \cite[Proposition 2.11]{HunekeUlrichThestructureoflinkage}.
\begin{lemma}
\label{lem: independent of generators for test submodule}
Let $\Lambda_1$ and $\Lambda_2$ be two sets of generators of $I$ and let $(S_1,J_1)$ and $(S_2,J_2)$ be generic links of $I$ with respect to $\Lambda_1$ and $\Lambda_2$ respectively. Then $\tau(\omega_{S_1/J_1})\supseteq \tau(I^c)\cdot(S_1/J_1)$ iff $\tau(\omega_{S_2/J_2})\supseteq \tau(I^c)\cdot(S_2/J_2)$.
\end{lemma}
\begin{proof}
By considering $\Lambda_1\cup \Lambda_2$, we can assume that $\Lambda_1\subseteq \Lambda_2$. By induction on the difference between the cardinality of $\Lambda_1$ and $\Lambda_2$, we may assume that $\Lambda_2$ has one more element than $\Lambda_1$, {\it i.e.} we may assume that $\Lambda_1=\{f_1,\dots,f_r\}$ and $\Lambda_2=\Lambda_1\cup \{f_{r+1}\}$.

Denote the height of $I$ by $c$. Let $\{u_{ij}\mid 1\leq i\leq c,1\leq j\leq r+1\}$ be indeterminates over $R$. Set $S_1=R[u_{ij}]_{1\leq i\leq c, 1\leq j\leq r}$ and $S_2=R[u_{ij}]_{1\leq i\leq c, 1\leq j\leq r+1}$. For $i=1,\ldots ,c$, set
\[g_i:=u_{i1}f_1+\cdots +u_{ir}f_{r}\] and \[h_i:=u_{i1}f_1+\cdots u_{i,r+1}f_{r+1}.\]
Then $J_1=((g_1,\ldots, g_c):_S IS)$ is the first generic link of $I$ with respect to $\Lambda_1$ and $J_2=((h_1,\ldots,h_c):_{S_2}IS_2)$ is the first generic link of $I$ with respect to $\Lambda_2$.

It is clear that $S_2=S_1[u_{1,r+1},\ldots, u_{c,r+1}]$. Since $f_{r+1}\in I$, we must have that $f_{r+1}=\sum_{j=1}^r a_jf_j$ for some $a_j\in R$.  Let $\varphi:S_2\to S_2$ be the automorphism given by the linear change of variables \[u_{ij}\mapsto u_{ij}+u_{i,r+1}a_j\] for $1\leq i\leq c$ and $1\leq j \leq r$ and \[u_{i,r+1}\mapsto u_{i,r+1}\] for $1 \leq i \leq c$.

We claim that $\varphi(J_1S_2)=J_2$ and we reason as follows. For $i=1,\ldots, c$, we have that
\[\varphi(g_i)=\sum_{j=1}^r(u_{ij}+u_{i,r+1}a_j)f_j=\sum_{j=1}^ru_{ij}f_j+u_{i,r+1}\sum_{j=1}^ra_jf_j=\sum_{j=1}^ru_{ij}f_j+u_{i,r+1}f_{r+1}=h_i.\]
Now since $S_1\hookrightarrow S_2$ is a faithfully flat extension, we have that
\[J_1S_2=((g_1,\ldots,g_c):_{S_1} IS_1)S_2=((g_1,\ldots,g_c)S_2:_{S_2} IS_2),\]
 and hence
%\begin{align*}
\[\varphi(J_1S_2) =\varphi((g_1,\ldots, g_c)S_2:_{S_2} IS_2)
=(\varphi(g_1,\ldots, g_c)S_2:_{S_2} \varphi(IS_2))
=((h_1,\ldots,h_c):_{S_2} IS_2)
=J_2.\]
%\end{align*}
Let $S^{\varphi}_2$ denote the $S_1$-algebra that is the same as $S_2$ as a ring and whose $S_1$-module structure is induced by $S_1\hookrightarrow S_2\xrightarrow{\varphi}S_2$. Then we have shown that $J_1\otimes_{S_1}S^{\varphi}_2=J_2$ and hence $S_1/J_1\otimes_{S_1}S^{\varphi}_2=S_2/J_2$. Combining Remarks \ref{rmk: trace does not change under polynomial ext} and \ref{remark--description of parameter test submodule}, one can check that
\[\tau(\omega_{S_1/J_1})\otimes_{S_1}S^{\varphi}_2=\tau(\omega_{S_1/J_1}\otimes_{S_1}S^{\varphi}_2)\]
where the right hand side is precisely $\tau(\omega_{S_2/J_2})$. Our lemma follows immediately since $S^{\varphi}_2$ is faithfully flat over $S_1$.
\end{proof}

The following lemma is also needed in the proof of Theorem \ref{theorem--parameter test submodules under generic linkage}.

\begin{lemma}
\label{claim finding exponents}
Let $c,r$ be positive integers such that $c = r$ or $c = r-1$.  Let $\beta =(\beta_1,\dots,\beta_r)$ be an element of $\bN^r$, where $\bN$ is the set of non-negative integers. Assume $\sum_i \beta_i = c(p^e-1)$. Then there exist $c$ elements $\alpha_1,...,\alpha_c$ in $\bN^r$ such that:
\begin{enumerate}
\item each $\alpha_i$ has at most two nonzero entries;
\item the sum of the entries of each $\alpha_i$ is $p^e-1$;
\item $\beta_j = \sum_i \alpha_{ij}$, where $\alpha_i=(\alpha_{i1},\dots,\alpha_{ir})$.
\end{enumerate}
\end{lemma}
\begin{proof}
We will induce on $r$. If $c=r=1$, then $\beta = (p^e-1)$ and we let $\alpha_1 = \beta$.  If $c=1, r=2$, we have $\beta = (\beta_1, \beta_2)$ where $\beta_1 + \beta_2 = p^e-1$ and we can let $\alpha_1= (\beta_1,\beta_2)$ and again (1)-(3) hold.

If $c=r$ and $\beta_1=\cdots=\beta_c=p^e-1$, then we can set $\alpha_i$ to be the vector with $p^e-1$ at $i$-th spot and $0$ elsewhere. Otherwise, there must be a $\beta_i < p^e-1$. Without loss of generality, we assume that $\beta_r<p^e-1$.

We claim that $\beta_{j} \geq p^e-1-\beta_r$ for some $j$ between $1$ and $r-1$, and we reason as follows. If $c=r$, then there must be a $j$ such that $\beta_j > p^e-1$, and hence $\beta_{j} \geq p^e-1-\beta_r$. Now assume that $c = r-1$. Suppose $\beta_{i} < p^e-1-\beta_r$ for all $i\leq r-1$, as then we would have:
\[
\sum_{i=1}^r\beta_i < (r-1)(p^e-1-\beta_r)+\beta_r \leq (r-2)(p^e-1-\beta_r)+(p^e-1) \leq (r-1)(p^e-1)=c(p^e-1)
\]
which contradicts the assumption that $\sum_{i=1}^r\beta_i=c(p^e-1)$. So, there is a $j$ between $1$ and $r-1$ such that $\beta_{j} \geq p^e-1-\beta_r$.

Set $\alpha_c := (0,\dots,0,p^e-1-\beta_r,0,\dots,\beta_r)$ where $p^e-1-\beta_r$ appears in the $j$-th spot. Consider \[(\beta_1,\dots,\beta_{j-1},\beta_j-(p^e-1-\beta_r),\beta_{j+1},\dots,\beta_{r-1}).\] This is an element of $\bN^{r-1}$ such that the sum of its entries is $(c-1)(p^e-1)$. By our induction hypotheses, there are $\gamma_1,\dots,\gamma_{c-1} \in \bN^{r-1}$ that satisfy (1), (2), and (3). For $1 \le i \le c-1$, setting $\alpha_i$ be $\gamma_i$ with a 0 added to the end completes the proof of our lemma.
\end{proof}

\begin{proof}[Proof of Theorem \ref{theorem--parameter test submodules under generic linkage}]
By Remark \ref{remark--geometrically linked}, $J$ is a minimal prime of $(g_1,\dots,g_c)$. Hence once we identify
$$\omega_{S/J}=\Hom_{S/(g_1,\dots,g_c)}(S/J, {S/(g_1,\dots,g_c)})=((g_1,\dots,g_c):J)\cdot (S/J)=I\cdot (S/J),$$
we know from Lemma \ref{lemma--restriction of trace} that
$$\Phi_{S/J}^e(F_*^e(-))=\Phi_{S/(g_1,\dots,g_c)}^e(F_*^e(-))|_{\omega_{S/J}}=\Tr_S^e(F_*^e(g_1^{p^e-1}\cdots g_c^{p^e-1}\cdot -))|_{I\cdot (S/J)}.$$
Next we notice that for every $1\leq k\leq r$, $(S/J)_{f_k}\cong R_{f_k}[u_{ij}|j\neq k]$ is regular. Hence for $N\gg0$, $f_k^N$ is a test element for $S/J$. Thus by Remark \ref{remark--description of parameter test submodule}, we have:
\begin{equation}
\label{equation--formula for parameter test submodule}
\tau(\omega_{S/J})=\sum_{e\geq 0} \Phi_{S/J}^e(F^e_*(f_k^N\cdot \omega_{S/J}))= \sum_{e\geq 0}\Tr_S^e(F_*^e(g_1^{p^e-1}\cdots g_c^{p^e-1}\cdot f_k^N \cdot IS))\cdot (S/J)
\end{equation}
Since $f_k\in I$ and $R$ is regular, by Remark \ref{remark--choose test element}, for $N\gg0$ we also have:
\begin{eqnarray}
\label{equation--formula for test ideal of pair}
\tau(I^c)\cdot (S/J)&=&\sum_{e\geq 0}\Tr_R^e(F_*^e((f_1,\dots,f_r)^{c(p^e-1)}\cdot f_k^N\cdot R))\cdot (S/J)
%&=& \sum_{e\geq 0}\Tr_R^e(F_*^e((f_1,\dots,f_t)^{c(p^e-1)}\cdot f_k^N\cdot I))\cdot (S/J) \notag
\end{eqnarray}
When we expand $g_1^{p^e-1}\cdots g_c^{p^e-1}$, it is easy to see from (\ref{equation--formula for parameter test submodule}) that $\tau(\omega_{S/J})$ can be generated by elements of the form
\begin{equation}
\label{equation--expression of generator of parameter test submodule}
\Tr_S^e\left(F_*^e\left(\binom{p^e-1}{\alpha_{11},\dots,\alpha_{1r}}\cdots\binom{p^e-1}{\alpha_{c1},\dots,\alpha_{cr}}f_1^{\beta_1}f_2^{\beta_2}\cdots f_r^{\beta_r}\prod_{\stackrel{1\leq i\leq c}{1\leq j\leq r}} u_{ij}^{\alpha_{ij}}\cdot f_k^N\cdot s\cdot \prod_{\stackrel{1\leq i\leq c}{1\leq j\leq r}} u_{ij}^{\gamma_{ij}}\right)\right)
\end{equation}
where $0\leq \alpha_{ij}\leq p^e-1$, $\beta_j=\sum_{i=1}^c\alpha_{ij}$, $\sum_{j=1}^r\beta_j=c(p^e-1)$ and $s \in I$. By definition of the trace map, this is equal to
$$\prod_{\stackrel{1\leq i\leq c}{1\leq j\leq r}} u_{ij}^{\frac{\alpha_{ij}+\gamma_{ij}-(p^e-1)}{p^e}}\cdot \Tr_R^e\left(F_*^e\left(\binom{p^e-1}{\alpha_{11},\dots,\alpha_{1r}}\cdots\binom{p^e-1}{\alpha_{c1},\dots,\alpha_{cr}}f_1^{\beta_1}f_2^{\beta_2}\cdots f_r^{\beta_r}\cdot f_k^N\cdot s\right)\right)$$
where $\frac{\alpha_{ij}+\gamma_{ij}-(p^e-1)}{p^e}$ denotes 0 if $\alpha_{ij}+\gamma_{ij}\not\equiv -1$ mod $p^e$. But it is clear that this element is in $\tau(I^c)\cdot S$ by expression (\ref{equation--formula for test ideal of pair}). This proves (1).

Next we prove (2).  By Lemma \ref{lem: independent of generators for test submodule} we can assume that $\tilde{I}=(f_1,\ldots, f_{c+1})$ is a reduction of $I$ (the case that $I$ has a reduction generated by $c$ elements is similar).  Hence by the arguments above, we have that, for $1\leq k \leq c$ and $N\gg 0$,
\[\tau(I^c)\cdot(S/J)=\tau(\tilde{I}^c)\cdot (S/J)=\sum_{e\geq 0}\Tr_R^e(F_*^e((f_1,\dots,f_{c+1})^{c(p^e-1)}\cdot f_k^{N+1}\cdot R))\cdot (S/J)\]
Given a generator $f_1^{\beta_1}\cdots f_{c+1}^{\beta_{c+1}}$ of $ (f_1,\dots,f_{c+1})^{c(p^e-1)}$, we can find $\alpha_1,\ldots, \alpha_c\in \N^{c+1}$ satisfying the conclusion of Lemma \ref{claim finding exponents}.  Then
\[\prod_{\stackrel{1\leq i\leq c}{1\leq j\leq c+1}}(u_{ij}f_j)^{\alpha_{ij}}=\prod_{\stackrel{1\leq i\leq c}{1\leq j\leq c+1}}u_{ij}^{\alpha_{ij}}\prod_{1 \leq j \leq c+1} f_j^{\beta_j}\]
appears with coefficient $\binom{p^e-1}{\alpha_{11},\dots,\alpha_{1,c+1}}\cdots\binom{p^e-1}{\alpha_{c1},\dots,\alpha_{c,c+1}}$ in the product $g_1^{p^e-1}\cdots g_c^{p^e-1}$.  Because each multinomial coefficient $\binom{p^e-1}{\alpha_{i1},\dots,\alpha_{i,c+1}}=\binom{p^e-1}{\alpha_{ij_i}}$ for some $j_i$ by Lemma \ref{claim finding exponents} (1)-(2), they are nonzero in $k$.

Each $\alpha_{ij}$ is less than $p^e$, so let
\[s'=\left(\prod_{\stackrel{1\leq i\leq c}{1\leq j\leq c+1}} u_{ij}^{p^e-1-\alpha_{ij}}\right)\left(\prod_{\stackrel{1\leq i\leq c}{c+2\leq j\leq r}} u_{ij}^{p^e-1}\right).\]  Then $\Tr_S^e(F^e_*( -\cdot s'))$ sends $\prod_{1\leq l \leq n} x_l^{p^e-1}\prod_{\stackrel{1\leq i\leq c}{1\leq j\leq c+1}}u_{ij}^{\alpha_{ij}}$ to 1 and all other basis elements to 0.  Hence,
\begin{align*}
&\Tr_S^e(F^e_*(g_1^{p^e-1}\cdots g_c^{p^e-1}\cdot f_k^{N+1} \cdot s' \cdot R))\\
=&\Tr_S^e\left(F^e_\ast\left(\binom{p^e-1}{\alpha_{11},\dots,\alpha_{1,c+1}}\cdots\binom{p^e-1}{\alpha_{c1},\dots,\alpha_{c,c+1}}\cdot \prod_{\stackrel{1\leq i\leq c}{1\leq j\leq r}}u_{ij}^{p^e-1}\prod_{j=1}^{c+1} f_j^{\beta_j} \cdot f_k^{N+1} R\right)\right)\\
%=&\Tr_S^e\left(F^e_\ast\left(\prod_{j=1}^{c+1} f_j^{\beta_j} \cdot f_k^{N+1}\cdot R\right)\right)\\
=&\Tr_R^e\left(F^e_*\left(\prod_{j=1}^{c+1} f_j^{\beta_j} \cdot f_k^{N+1}\cdot R\right)\right).
\end{align*}
In particular, \[\Tr_R^e\left(F^e_*\left(\prod_{j=1}^{c+1} f_j^{\beta_j} \cdot f_k^{N+1}\cdot R\right)\right)\cdot (S/J)=\Tr_S^e(F^e_*(g_1^{p^e-1}\cdots g_c^{p^e-1}\cdot f_k^N\cdot f_k s' R))\cdot (S/J)\subseteq \tau(\omega_{S/J})\] for every generator $\prod_{j=1}^{c+1} f_j^{\beta_j}$ of $(f_1,\ldots, f_{c+1})^{c(p^e-1)}$, where the second inclusion follows from expression (\ref{equation--formula for parameter test submodule}). Therefore we have
\begin{align*}
\tau(I^c)\cdot (S/J) &=\tau(\tilde{I}^c)\cdot (S/J)\\
&=\sum_{e\geq 0}\Tr_R^e(F_*^e((f_1,\dots,f_{c+1})^{c(p^e-1)}\cdot f_k^{N+1}\cdot R))\cdot (S/J)\\
%&\subseteq \sum_{e\geq 0}\Tr_S^e(F^e_*(g_1^{p^e-1}\cdots g_c^{p^e-1}\cdot f_k^N\cdot f_k s' R))\cdot (S/J)\\
&\subseteq \tau(\omega_{S/J}).\qedhere\end{align*}

\end{proof}

%\Zhang{The diagram is not used in the paper; we only need the isomorphism in the bottom which is not our result. To keep it, we need to say a few words about the commutativity.}

%\begin{remark}
%A consequence of our result is that we get a commutative diagram
%\[\xymatrix{
%\tau(\omega_{S/J}) \ar[r] \ar@{^{(}->}[d]& \tau(I^c)\cdot (S/J) \ar@{^{(}->}[d]\\
%\omega_{S/J} \ar[r]^-\cong & I\cdot (S/J)
%}\]
%analogous to the one in \cite[Proposition~3.2]{NiuSingularitiesofgenericlinkage},
%where the bottom isomorphism is canonical (see the proof of Theorem \ref{theorem--parameter test submodules under generic linkage}), the left inclusion is induced by the trace map, and the right inclusion is induced by $\tau(I^c)\hookrightarrow I$ as in (\ref{equation--containment of test ideal 1}). The map on the top is the canonical isomorphism $\tau(\omega_{S/J}) \cong \tau(I \cdot (S/J))$ followed by the equality we proved between $\tau(I \cdot (S/J))$ and $\tau(I^c) \cdot (S/J)$, and hence the diagram commutes. %check this
%\end{remark}

\begin{remark}
The proof of Theorem \ref{theorem--parameter test submodules under generic linkage} (2) requires the minimal reduction be generated by at most $c+1$ elements. If not, then we are not in the case of Lemma \ref{claim finding exponents} and it may be the case that there are always at least three nonzero entries in some $\alpha_i$. Consequently, multinomial coefficients must be taken into consideration.
\end{remark}

%\todo{(Rebecca): When $r \geq c+2$, there will definitely be some set of $\beta_i$ that require us to use multinomial coefficients. (In the proof of the theorem, we need to deal with all possible sets of $\beta_i$, so the example from the previous version of this paper will come up.)}

\begin{corollary}
\label{corollary--criterion for F-rationality}
With the notation as in Definition \ref{definition--generic linkage} and the assumptions as in Theorem \ref{theorem--parameter test submodules under generic linkage} (2), $\tau(\omega_{S/J})=\omega_{S/J}$ if and only if $\tau(I^c)=I$. In particular, $S/J$ has $F$-rational singularities if and only if $S/J$ is Cohen-Macaulay and $\tau(I^c)=I$.
\end{corollary}
\begin{proof}
If $\tau(I^c)=I$, then Theorem \ref{theorem--parameter test submodules under generic linkage} immediately implies $\tau(\omega_{S/J})=\omega_{S/J}$.

Conversely, assume that $\tau(I^c)\neq I$ and $\tau(\omega_{S/J})=\omega_{S/J}$. Since $\tau(I^c)$ is always contained in $I$ by (\ref{equation--containment of test ideal 1}), at least one of the generators of $I$ is not in $\tau(I^c)$, say $f_1$. From Theorem \ref{theorem--parameter test submodules under generic linkage}, we can see that $\tau(I^c)S+J=IS+J$; hence $f_1\in \tau(I^c)S+J$. Thus, there are elements $a\in \tau(I^c)S$ and $b\in J$ such that $f_1=a+b$. (Note that $b\neq 0$.) Then we have $f_1-a\in J$ which implies that $(f_1-a)f_1\in (g_1,\dots,g_c)$. However, this is impossible because of the degrees in the $u_{ij}$. This is a contradiction.

The last assertion is clear because $S/J$ is $F$-rational if and only if $S/J$ is Cohen-Macaulay and $\tau(\omega_{S/J})=\omega_{S/J}$.
\end{proof}

\begin{corollary}
With the notation as in Definition \ref{definition--generic linkage} and the assumptions as in Theorem \ref{theorem--parameter test submodules under generic linkage} (2), if the pair $(R, I^c)$ is $F$-pure and $R/I$ is Cohen-Macaulay, then $S/J$ is $F$-rational. In particular, if $R/I$ is an $F$-pure complete intersection, then $S/J$ is $F$-rational.
\end{corollary}
\begin{proof}
By Lemma \ref{lemma--Ein's lemma in char p}, $(R, I^c)$ is $F$-pure implies $\tau(I^c)=I$. The first statement thus follows from Corollary \ref{corollary--criterion for F-rationality}. Finally, it is well known that when $R/I$ is an $F$-pure complete intersection, the pair $(R, I^c)$ is $F$-pure. This follows from a Fedder type criterion (\cite[Lemma 3.9]{TakagiFsingularitiesofpairsandinversionofadjunction} and others).
\end{proof}

We can recover \cite[Corollary 3.4]{NiuSingularitiesofgenericlinkage} in the complete intersection and almost complete intersection cases.
\begin{corollary}
Let $I=(f_1,\dots,f_r)$ be an ideal of $\bC[x_1,\dots,x_n]$ and let $c$ be the codimension of $I$. Let $S$ and $J$ be in Definition \ref{definition--generic linkage}. Assume that $r\leq c+1$. Then $S/J$ has rational singularities if and only if $S/J$ is Cohen-Macaulay and $\scr{I}(I^c)=I$, where $\scr{I}(I^c)$ is the multiplier ideal of $I^c$.
\end{corollary}
\begin{proof}
By \cite{SmithFRatImpliesRat} and \cite{HaraRatImpliesFRat}, $S/J$ has rational singularities if and only if its reduction $(S/J)_p$ is $F$-rational for all $p\gg 0$. It is easy to see that, for $p\gg 0$, the reduction $J_p$ of $J$ is a generic link of the reduction $I_p$ of $I$. Hence, $S/J$ has rational singularities if and only if $(S/J)_p$ is Cohen-Macaulay and $\tau(I^c_p)=I_p$ for $p\gg 0$ by Corollary \ref{corollary--criterion for F-rationality}. On the other hand, it was proved in \cite{HaraYoshidaGeneralizedTestIdeals} that $\scr{I}(I^c)_p=\tau(I^c_p)$ for all $p\gg 0$. Therefore, we have $S/J$ has rational singularities if and only if $(S/J)_p$ is Cohen-Macaulay and $\scr{I}(I^c)_p=I_p$ for $p\gg 0$. This completes the proof.
\end{proof}

%%%%%%%%%%%%%%%%%%%%%%%%%%%%%%%%%%%%%%%%%%%%%%%%%%%%%%%%%%%%%%%%%%%%%%%%%%%%%%%%%
%%%%%%%%%%%%%%%%%%%%%%%%%%%%%%%%%%%%%%%%%%%%%%%%%%%%%%%%%%%%%%%%%%%%%%%%%%%%%%%%%
\section{Behavior of $F$-pure threshold under generic linkage}
\label{section: FPT under linkage}
%%%%%%%%%%%%%%%%%%%%%%%%%%%%%%%%%%%%%%%%%%%%%%%%%%%%%%%%%%%%%%%%%%%%%%%%%%%%%%%%%
%%%%%%%%%%%%%%%%%%%%%%%%%%%%%%%%%%%%%%%%%%%%%%%%%%%%%%%%%%%%%%%%%%%%%%%%%%%%%%%%%
%The $F$-pure threshold is a measure of the singularities of an ideal $I$ of a ring $R$.  In the situation of Definition \ref{definition--generic linkage}, we wish to compare the $F$-pure threshold of $I$ to that of its first generic link $J$.
In this section we investigate behaviors of $F$-pure thresholds under generic linkages. We begin with an easy lemma.

\begin{lemma}
\label{lem: independent of generators for FPT}
Let $R=k[x_1,\dots,x_n]$ be a polynomial ring over a perfect field of characteristic $p$ and $I$ be an equi-dimensional and unmixed ideal of $R$. Let $\Lambda_1$ and $\Lambda_2$ be 2 sets of generators of $I$ and let $(S_i,J_i)$ be the generic link with respect to $\Lambda_i$ (i=1,2). Then \[\fp_{S_1}(J_1)=\fp_{S_2}(J_2).\]
\end{lemma}

\begin{proof}
As in the proof of Lemma \ref{lem: independent of generators for test submodule}, we can assume that $\Lambda_1=\{f_1,\dots,f_r\}$ and $\Lambda_2=\{f_1,\dots,f_r,f_{r+1}\}$. Let $\varphi$ and $S^{\varphi}_2$ be the same as in the proof of Lemma \ref{lem: independent of generators for test submodule}. It is straightforward to check that
\[\tau(J^t_1)\otimes_{S_1}S^{\varphi}_2=\tau(J^t_2)\]
for each nonnegative real number $t$. Our lemma follows immediately.
\end{proof}

\begin{remark}
\label{remark: index of perfect fields}
Let $k\subseteq K$ be an extension of perfect fields and let $R=k[x_1,\dots,x_n]$ and $T=K[x_1,\dots,x_n]$. Since $\Hom_{R}(R^{1/p^e},R)$ and $\Hom_{T}(T^{1/p^e},T)$ are generated by the same projection, we have $\tau_R(I^t)=\tau_{T}((IT)^t)$ ({\it c.f.} \cite[Remark 2.18]{BlickleMustataSmithDiscretenessRationalityF-Thresholds}).
\end{remark}

\begin{theorem}
\label{theorem: FPT lower bound on g's}
Let $R=k[x_1,\dots,x_n]$ be a polynomial ring over a perfect field of characteristic $p$ and $I$ be an equi-dimensional and unmixed ideal of height $c$ in $R$. Assume that $I=(f_1,\ldots,f_s)$ and that $I$ has a reduction $\tilde{I}$ generated by $r$ elements. Let $S=R[u_{ij}]_{1 \le i \le c, 1 \le j \le s}$ be a polynomial ring over $R$. For $1 \le i \le c$, let \[g_i=u_{i1}f_1+u_{i2}f_2+\ldots+u_{is}f_s.\] Then $\fp_S(g_1,\ldots, g_c)\geq \frac{c}{r}\fp_{R}(I)$.
\end{theorem}
\begin{proof}
%Adding the generators of $\tilde{I}$ to those of $I$ if necessary, we can assume that the generators of $\tilde{I}$ are a subset of $f_1,\ldots,f_s$.
By Lemma \ref{lem: independent of generators for FPT}, we can add the generators of $\tilde{I}$ to those of $I$ and then assume that $\tilde{I}=(f_1,\ldots,f_r)$. Since $\tilde{I}$ is a reduction of $I$, it follows from \cite[Proposition 2.2(6)]{TakagiWatanabeFpurethresholds} that $\fp_R(I)=\fp_R(\tilde{I})$. Hence it suffices to show that $\tau_R(\tilde{I}^t)=R$ implies $\tau_S((g_1,\ldots, g_c)^{\frac{ct}{r}})=S$ for any positive real number $t<c$. To this end, assume that $\tau_R(\tilde{I}^t)=R$. By Remark \ref{remark: index of perfect fields}, we may assume that $k$ is algebraically closed.

We wish to show that $\tau_S((g_1,\ldots, g_c)^{\frac{ct}{r}})=S$. Suppose otherwise and we seek a contradiction. There is a maximal ideal $\m$ of $S$ such that $\tau_S((g_1,\ldots, g_c)^{\frac{ct}{r}})\subseteq \m$. Since $k$ is algebraically closed, we can write $\m=(x_1-a_1,\ldots,x_n-a_n,u_{11}-b_{11},\ldots,u_{cr}-b_{cr})$ for some $a_i,b_{ij}\in k$. Set $\n=(x_1-a_1,\ldots,x_n-a_n)$. Since $\tau_R(\tilde{I}^t)=R$, there exist an integer $e$, an $R$-linear map $\phi\in \Hom_R(R^{1/p^e},R)$, and nonnegative integers $\alpha_1,\ldots,\alpha_r$ with $\sum_i \alpha_i=\lceil tp^e \rceil$ such that $\phi(f^{\alpha_1/p^e}_1\cdots f^{\alpha_r/p^e}_r)\notin \n$.

%***New content 2/7/18 ***

At this point we show that each $f_j\in \n$, and therefore $\alpha_j\leq p^e-1$ for all $j$.  Indeed, let $e\geq 1$ such that $p^e\geq c/(c-t)$ and let  $\psi:S^{1/p^e}\to S$ send the basis element $u_{1j}^{(p^e-1)/p^e}u_{2j}^{(p^e-1)/p^e}\cdots u_{cj}^{(p^e-1)/p^e}$ to $1$ and all other basis elements $x_\ell^{a_\ell/p^e}u_{ij}^{b_{ij}/p^e}$ to $0$.  Now  $f_j^cg_1^{p^e-1}g_2^{p^e-1}\cdots g_c^{p^e-1}\in (g_1,\cdots,g_c)^{\lceil (ct/r)p^e\rceil}$, because $(ct/r)p^e\leq tp^e\leq c(p^e-1)$. Therefore
\begin{eqnarray*}
f_j^c&=&\psi(f_j^{c/p^e}u_{1j}^{(p^e-1)/p^e}u_{2j}^{(p^e-1)/p^e}\cdots u_{cj}^{(p^e-1)/p^e}f_j^{c(p^e-1)/p^e})\\
&=&\psi(f_j^{c/p^e}g_1^{(p^e-1)/p^e}g_2^{(p^e-1)/p^e}\cdots g_c^{(p^e-1)/p^e})\\
&\subseteq& \psi(((g_1,\cdots,g_c)^{\lceil (ct/r)p^e\rceil})^{1/p^e})\subseteq \m
\end{eqnarray*}
by our choice of $\psi$ and the assumption that $\tau_S((g_1,\ldots, g_c)^{\frac{ct}{r}})\subseteq \m$. It follows that $f_j\in \m\cap R=\n$.

%***End of new content***
% it is equivalent to showing that $\tau_{S_{\m}}((g_1,\ldots, g_c)^{\frac{ct}{r}})=S_{\m}$ for each maximal ideal $\m$ of $S$. According to \cite[Remark 2.18]{BlickleMustataSmithDiscretenessRationalityF-Thresholds}, we can assume that $k$ is algebraically closed. Thus, we may write $\m=(x_1-a_1,\ldots,x_n-a_n,u_{11}-b_{11},\ldots,u_{cr}-b_{cr})$ for some $a_i,b_{ij}\in k$. Set $\n=(x_1-a_1,\ldots,x_n-a_n)$. Then $\tau_{R_{\n}}(\tilde{I}^t)=R_{{n}}$. Therefore, there exists an $R_{\n}$-linear map $\phi:R_{\n}^{1/p^e}\to R_{\n}$ and an element $z\in \tilde{I}^{\lceil tp^e\rceil }$ such that $\phi(z^{1/p^e})=1$. Write
%\[z=\sum_{0\leq a_i\leq p^e-1;\sum_ia_i=\lceil tp^e \rceil}z_{\underline{a}}f^{a_1}_1\cdots f^{a_r}_r,\ \underline{a}=(a_1,\dots,a_r).\]
%Since $R_{\n}$ is a local ring, $\phi(z_{\underline{a}}^{1/p^e}f^{a_1/p^e}_1\cdots f^{a_r/p^e}_r)$ is a unit for at least one choice of $\underline{a}=(a_1,\dots,a_r)$. Fix such an $\underline{a}$. We may assume that $\phi(z_{\underline{a}}^{1/p^e}f^{a_1/p^e}_1\cdots f^{a_r/p^e}_r)$=1.

Without loss of generality, we may assume that $p^e-1\geq \alpha_1\geq \alpha_2\geq \cdots \geq \alpha_r$. Consequently,
\[\alpha_1+\cdots +\alpha_c\geq \left\lceil{\frac{c}{r}(\alpha_1+\cdots +\alpha_r)}\right\rceil=\left\lceil\frac{c}{r}\lceil tp^e\rceil\right\rceil\geq \left\lceil\frac{c}{r}tp^e\right\rceil\]

Let $\phi_{\underline{\alpha}}=\phi(f^{\alpha_{c+1}/p^e}_{c+1}\cdots f^{\alpha_r/p^e}_r\cdot -)$, {\it i.e.} pre-multiplication by $f^{\alpha_{c+1}/p^e}_{c+1}\cdots f^{\alpha_r/p^e}_r$ followed by the application of $\phi$. It is clear that $\phi_{\underline{\alpha}}:R^{1/p^e}\to R$ is an $R$-linear map and that $\phi_{\underline{\alpha}}(f^{\alpha_1/p^e}_1\cdots f^{\alpha_c/p^e}_c)\notin \n$. We can extend $\phi_{\underline{\alpha}}$ to an $S$-linear map $\psi_{\underline{\alpha}}:R^{1/p^e}[u_{ij}]\to S=R[u_{ij}]$ that sends each $u_{ij}$ to itself and restricts to $\phi_{\underline{\alpha}}$ on $R^{1/p^e}$.

It is clear that $S^{1/p^e}=R^{1/p^e}[u^{1/p^e}_{ij}]$ is a free $R^{1/p^e}[u_{ij}]$-module with a basis $\{\prod_{0\leq b_{ij}\leq p^e-1}u^{b_{ij}/p^e}_{ij}\}$. Let $\pi_{\underline{\alpha}}:R^{1/p^e}[u^{1/p^e}_{ij}]\to R^{1/p^e}[u_{ij}]$ be the projection that sends $u^{\alpha_1/p^e}_{11}\cdots u^{\alpha_c/p^e}_{cc}$ to 1 and all other basis elements to 0.

Let $\theta_{\underline{\alpha}}$ be the composition of $S^{1/p^e}\xrightarrow{\pi_{\underline{\alpha}}}R^{1/p^e}[u_{ij}]\xrightarrow{\psi_{z_{\underline{a}}}} S$. It is clear that $\theta_{\underline{\alpha}}$ is $S$-linear. By the construction of $\pi_{\underline{\alpha}}$, it is straightforward to check that
\[\theta_{\underline{a}}(g^{\alpha_1/p^e}_1\cdots g^{\alpha_c/p^e}_c)=\theta_{\underline{\alpha}}((u_{11}f_1)^{\alpha_1/p^e}\cdots (u_{cc}f_c)^{\alpha_c/p^e})=\phi(f^{\alpha_1/p^e}_1\cdots f^{\alpha_r/p^e}_r).\]

Since $\phi(f^{\alpha_1/p^e}_1\cdots f^{\alpha_r/p^e}_r)$ in $R$ but not in $\n=(x_1-a_1,\ldots,x_n-a_n)$, we must have
\[\phi(f^{\alpha_1/p^e}_1\cdots f^{\alpha_r/p^e}_r)\notin \m=(x_1-a_1,\ldots,x_n-a_n,u_{11}-b_{11},\ldots,u_{cr}-b_{cr}),\] a contradiction to the assumption that $\tau_S((g_1,\ldots, g_c)^{\frac{ct}{r}})\subseteq \m$ (note that $g^{\alpha_1}_1\cdots g^{\alpha_c}_c\in (g_1,\ldots, g_c)^{\lceil \frac{ct}{r} p^e\rceil}$).
\end{proof}

We have some immediate corollaries.

\begin{corollary}
\label{corollary -- generic link lower bound}
Let $R=k[x_1,\dots,x_n]$ be a polynomial ring over a perfect field of characteristic $p$ and $I$ be an equi-dimensional and unmixed ideal of height $c$ in $R$. Let $J$ be a generic link of $I$ in $S=R[u_{ij}]$. The following hold:
\begin{enumerate}
\item If $I$ has a reduction generated by $r$ elements, then $\fp_{S}(J)\geq \frac{c}{r}\fp_{R}(I)$.
\item If $I$ has a reduction generated by $c$ elements, in particular if $I$ is a complete intersection, then $\fp_{S}(J)\geq \fp_R(I)$.
\item $\fp_{S}(J)\geq \frac{c}{n}\fp_{R}(I)$ (note $n=\dim(R)$).
\end{enumerate}
\end{corollary}

\begin{proof}
To prove (1), note that since $(g_1,\dots,g_c)\subseteq J$, we have $\fp_S(J)\geq \fp_S(g_1,\dots,g_c)$. Theorem \ref{theorem: FPT lower bound on g's} then completes the proof.

(2) is an immediate consequence of (1).

(3) By Remark \ref{remark: index of perfect fields}, passing to the algebraic closure of $k$ doesn't affect $\fp_R(I)$ and $\fp_S(J)$. Hence we can assume that $k$ is algebraically closed and hence is infinite. \cite[Theorem]{LyubeznikReductionidealsinPolynomialRings} asserts that each ideal $I$ admits a reduction generated by $n$ elements. We are done by (1).
\end{proof}

\bibliographystyle{alpha}
\bibliography{CommonBib}

\begin{thebibliography}{BSTZ10}

\bibitem[BMS08]{BlickleMustataSmithDiscretenessRationalityF-Thresholds}
Manuel Blickle, Mircea Musta{\c{t}}{\v{a}}, and Karen~E. Smith.
\newblock Discreteness and rationality of {$F$}-thresholds.
\newblock {\em Michigan Math. J.}, 57:43--61, 2008.
\newblock Special volume in honor of Melvin Hochster.

\bibitem[BST15]{BlickleSchwedeTuckerF-singularitiesvialterations}
Manuel Blickle, Karl Schwede, and Kevin Tucker.
\newblock {$F$}-singularities via alterations.
\newblock {\em Amer. J. Math.}, 137(1):61--109, 2015.

\bibitem[BSTZ10]{BlickleSchwedeTakagiZhangDiscretenessRationality}
Manuel Blickle, Karl Schwede, Shunsuke Takagi, and Wenliang Zhang.
\newblock Discreteness and rationality of {$F$}-jumping numbers on singular
  varieties.
\newblock {\em Math. Ann.}, 347(4):917--949, 2010.

\bibitem[CU02]{ChardinUlrichLiaisonandCMregularity}
Marc Chardin and Bernd Ulrich.
\newblock Liaison and {C}astelnuovo-{M}umford regularity.
\newblock {\em Amer. J. Math.}, 124(6):1103--1124, 2002.

\bibitem[EHU04]{EisenbudHunekeUlrichHeightsofidealsofminors}
David Eisenbud, Craig Huneke, and Bernd Ulrich.
\newblock Heights of ideals of minors.
\newblock {\em Amer. J. Math.}, 126(2):417--438, 2004.

\bibitem[Fed83]{FedderFPureRat}
Richard Fedder.
\newblock {$F$}-purity and rational singularity.
\newblock {\em Trans. Amer. Math. Soc.}, 278(2):461--480, 1983.

\bibitem[Har98]{HaraRatImpliesFRat}
Nobuo Hara.
\newblock A characterization of rational singularities in terms of injectivity
  of {F}robenius maps.
\newblock {\em Amer. J. Math.}, 120(5):981--996, 1998.

\bibitem[Har01]{HaraGeometricinterpretation}
Nobuo Hara.
\newblock Geometric interpretation of tight closure and test ideals.
\newblock {\em Trans. Amer. Math. Soc.}, 353(5):1885--1906, 2001.

\bibitem[HT04]{HaraTakagiOnGeneralizationTestideals}
Nobuo Hara and Shunsuke Takagi.
\newblock On a generalization of test ideals.
\newblock {\em Nagoya Math. J.}, 175:59--74, 2004.

\bibitem[HU85]{HunekeUlrichDivisorclassgroupsdeformations}
Craig Huneke and Bernd Ulrich.
\newblock Divisor class groups and deformations.
\newblock {\em Amer. J. Math.}, 107(6):1265--1303 (1986), 1985.

\bibitem[HU87]{HunekeUlrichThestructureoflinkage}
Craig Huneke and Bernd Ulrich.
\newblock The structure of linkage.
\newblock {\em Ann. of Math. (2)}, 126(2):277--334, 1987.

\bibitem[HU88]{HunekeUlrichAlgebraiclinkage}
Craig Huneke and Bernd Ulrich.
\newblock Algebraic linkage.
\newblock {\em Duke Math. J.}, 56(3):415--429, 1988.

\bibitem[HY03]{HaraYoshidaGeneralizedTestIdeals}
Nobuo Hara and Ken-Ichi Yoshida.
\newblock A generalization of tight closure and multiplier ideals.
\newblock {\em Trans. Amer. Math. Soc.}, 355(8):3143--3174 (electronic), 2003.

\bibitem[Lyu86]{LyubeznikReductionidealsinPolynomialRings}
Gennady Lyubeznik.
\newblock A property of ideals in polynomial rings.
\newblock {\em Proc. Amer. Math. Soc.}, 98(3):399--400, 1986.

\bibitem[Mat86]{Matsumura86}
Hideyuki Matsumura.
\newblock {\em Commutative ring theory}, volume~8 of {\em Cambridge Studies in
  Advanced Mathematics}.
\newblock Cambridge University Press, Cambridge, 1986.
\newblock Translated from the Japanese by M. Reid.

\bibitem[Niu14]{NiuSingularitiesofgenericlinkage}
Wenbo Niu.
\newblock Singularities of generic linkage of algebraic varieties.
\newblock {\em Amer. J. Math.}, 136(6):1665--1691, 2014.

\bibitem[Sch11]{SchwedeTestidealsinnonQ-Gorensteinrings}
Karl Schwede.
\newblock Test ideals in non-{$\Bbb{Q}$}-{G}orenstein rings.
\newblock {\em Trans. Amer. Math. Soc.}, 363(11):5925--5941, 2011.

\bibitem[Smi97]{SmithFRatImpliesRat}
Karen~E. Smith.
\newblock {$F$}-rational rings have rational singularities.
\newblock {\em Amer. J. Math.}, 119(1):159--180, 1997.

\bibitem[ST12]{SchwedeTuckerAsurveyoftestideals}
Karl Schwede and Kevin Tucker.
\newblock A survey of test ideals.
\newblock In {\em Progress in commutative algebra 2}, pages 39--99. Walter de
  Gruyter, Berlin, 2012.

\bibitem[Tak04]{TakagiFsingularitiesofpairsandinversionofadjunction}
Shunsuke Takagi.
\newblock F-singularities of pairs and inversion of adjunction of arbitrary
  codimension.
\newblock {\em Invent. Math.}, 157(1):123--146, 2004.

\bibitem[TW04]{TakagiWatanabeFpurethresholds}
Shunsuke Takagi and Kei-ichi Watanabe.
\newblock On {F}-pure thresholds.
\newblock {\em J. Algebra}, 282(1):278--297, 2004.

\end{thebibliography}
\end{document}